\documentclass[11pt]{article}
\usepackage[english]{babel}
\usepackage[utf8]{inputenc}
\usepackage{amsmath}
\usepackage{amsthm}
\usepackage{amssymb}

\usepackage{comment}
\usepackage{enumerate}
\usepackage[dvips]{color}
\usepackage[hidelinks = true]{hyperref}
\usepackage{doi}
\usepackage{cleveref}
\usepackage{xcolor}

\usepackage{amscd}

\usepackage{amsmath}
\usepackage{amsthm}
\usepackage{amssymb}

\newtheorem{theorem}{Theorem}[section]
\newtheorem{definition}[theorem]{Definition}

\newtheorem{lemma}[theorem]{Lemma}

\theoremstyle{definition}
\newtheorem{remark}[theorem]{Remark}

\newcommand{\rr}{\mathbb{R}}

\newcommand{\cc}{\mathbb{C}}
\newcommand{\hh}{\mathbb{H}}

\newcommand{\pp}{\partial}

\title{\bf Infinite order differential operators acting on entire hyperholomorphic functions   }

\author{
D. Alpay
\footnote{Schmid College of Science and Technology, Chapman University, Orange 92866, CA, US},
F. Colombo $^\dag$,
S. Pinton$^\dag$,
I. Sabadini
 \footnote{Politecnico di
Milano, Dipartimento di Matematica, Via E. Bonardi, 9 20133 Milano, Italy},
 D.C. Struppa
\footnote{Donald Bren Distinguished Presidential Chair in Mathematics, Chapman University, Orange 92866, CA, US}}

\begin{document}
\maketitle

\begin{abstract}
Infinite order differential operators appear in different fields of Mathematics and Physics and in the last decades
they turned out to be of fundamental importance in the study of the evolution of superoscillations as initial datum for Schr\"odinger  equation.
Inspired by the operators arising
in quantum mechanics, in this paper we investigate the continuity of a class of infinite order differential operators
acting on spaces of entire hyperholomorphic functions.
The two classes of hyperholomorphic functions, that constitute a natural extension of functions of
 one complex variable to functions of paravector variables are illustrated by
the Fueter-Sce-Qian mapping theorem.
We show that, even though the two notions of hyperholomorphic functions are quite different from each other,
 entire hyperholomorphic functions with exponential bounds
play a crucial role in the continuity of infinite order differential operators
acting on these two classes of entire hyperholomorphic functions.
We point out the remarkable fact that the exponential function of a paravector variable
is not in the kernel of the Dirac operator
but entire  monogenic functions with exponential bounds play an important role in the theory.
\end{abstract}
\vskip 1cm

\section{Introduction}

Infinite order differential operators
 turned out to be of fundamental importance in the study of the evolution of superoscillations as initial datum for Schr\"odinger  equation.
To study the evolution of superoscillatory functions under Schr\"odinger equation
is highly nontrivial and a natural functional setting is the space of entire functions with growth conditions,
for more details see the monograph \cite{acsst5} and \cite{Be19}.
In fact, the Cauchy problem for Schr\"odinger equation with superoscillatory initial datum leads
to  infinite order differential operators of the  type
$$
\mathcal{U}(t,x;D_x)=\sum_{m=1}^\infty u_m(t,x)D_x^m,
$$
where the coefficients $u_m(t,x)$ depend on the Green function of the Schr\"odinger equation with the potential $V$,
and $t$ and $x$ are the time and the space variables. According to the structure of the Green function
 the coefficients  $u_m(t,x)$ satisfy given growth conditions.
For some potentials $V$ we are forced to consider infinite order differential operators $\mathcal{P}(t,x;D_\xi)$ depending on
 an auxiliary complex variable $\xi$
$$
\mathcal{P}(t,x;D_\xi)=\sum_{m=1}^\infty u_m(t,x)D_\xi^m,
$$
with coefficients $u_m(t,x)$ that depend on $V$.
The continuity properties  of the operators $\mathcal{U}(t,x;D_x)$ or $\mathcal{P}(t,x;D_\xi)$ acting
 on the spaces of entire functions with exponential bounds are
  the heart of the study of the evolution of superoscillatory initial datum in quantum mechanics.
For $p\geq 1$ the natural spaces on which such the operators $\mathcal{U}(t,x;D_x)$ and $\mathcal{P}(t,x;D_\xi)$ act
are the spaces of entire functions with either order lower than $p$ or order equal to $p$ and finite type.
They  consist of entire functions $f$ for which there exist constants $B, C >0$ such that
$
|f(z)|\leq C e^{B|z|^p}.
$

\medskip
This paper is devoted to a double audience: for researchers working in complex and hypercomplex analysis and for experts working in the area of infinite order differential operators.
More precisely,  we investigate the continuity of a class of infinite order differential operators
acting on spaces of entire hyperholomorphic functions.
There are two main classes of hyperholomorphic functions
that constitute the natural extension of functions of
 one complex variable to functions of paravector variables as illustrated by
the Fueter-Sce-Qian mapping theorem, as it is recalled in the last section of this paper.
From this theorem naturally emerge
the slice hyperholomorphic functions  and the functions in the kernel of the Dirac operator that are called monogenic functions.
We show that, even though the two notions of hyperholomorphic functions are quite different from each other,
 and the exponential function is not in the kernel of the Dirac operator, hyperholomorphic functions with exponential bounds
play a crucial role in the continuity of a class of infinite order differential operators
in the hypercomplex settings. The complex version of these results where studied in the paper \cite{ACSS18}.

\medskip
In the following
we denote by $\rr_n$  the real Clifford algebra over $n$ imaginary units $e_1,\ldots ,e_n$
satisfying the relations $e_\ell e_m+e_me_\ell=0$,\  $\ell\not= m$, $e_\ell^2=-1.$
 An element $(x_0,x_1,\ldots,x_n)\in \mathbb{R}^{n+1}$  will be identified with the element
$x=x_0+\underline{x}=x_0+ \sum_{\ell=1}^nx_\ell e_\ell\in\mathbb{R}_n$ which is called paravector.
We denote by $\mathbb{S}$ the sphere
$
\mathbb{S}=\{ \underline{x}=e_1x_1+\ldots +e_nx_n\ | \  x_1^2+\ldots +x_n^2=1\}$,
we observe that
for $\mathbf{j}\in\mathbb{S}$ we obviously have $\mathbf{j}^2=-1$, with the imaginary unit $\mathbf{j}$ we obtain the complex plane
$C_\mathbf{j}$ whose elements are of the form $u+\mathbf{j}v$ for $u$, $v\in \mathbb{R}$.
 The first class of hyperholomorphic functions we consider are called slice hyperholomorphic (or slice monogenic functions) and are defined as follows.
 Let $U\subseteq \mathbb{R}^{n+1}$ be an axially symmetric open set and let
 $\mathcal{U} = \{ (u,v)\in\rr^2: u+ \mathbb{S} v\subset U\}$. A function $f:U\to \mathbb{R}_n$ is called a left slice function, if it is of the form
 \[
 f(q) = f_{0}(u,v) + \mathbf{j}f_{1}(u,v)\qquad \text{for }  \ \ \  (u,v)\in\mathcal{U}
 \]
 where $q=u+\mathbf{j}v$ and the
 two functions $f_{0},f_{1}: \mathcal{U} \to \mathbb{R}_n$  satisfy the compatibility conditions
\begin{equation}\label{CCondmon}
f_{0}(u,-v) = f_{0}(u,v),\qquad f_{1}(u,-v) = -f_{1}(u,v).
\end{equation}
If in addition $f_{0}$ and $f_{1}$ satisfy the Cauchy-Riemann-equations then $f$ is called left slice monogenic functions.
A similar notion is given in the sequel for right slice monogenic functions.

The second class of hyperholomorphic functions that we consider consists of the monogenic functions.
Left monogenic functions are define as those functions
 $f:U\subseteq \mathbb{R}^{n+1}\to \mathbb{R}_n$
that are $C^1$ and that
 are in the kernel of the Dirac operator $\mathcal D$  defined as:
$$
 \mathcal D:= \frac{\partial}{\partial_{x_0}}+\sum_{i=1}^n e_i \frac{\partial}{\partial_{x_i}},
$$
that is, $ \mathcal Df(x)=0$.
Also in this case there exists the notion of right monogenic functions.

\medskip
Let us point out the main differences between the two function theories in order to
appreciate the analogies with respect to infinite order differential operators.

\medskip
(A) The pointwise product of two hyperholomorphic functions,  in general, is not
hyperholomorphic, so we need to define the product in a way that preserves the hyperholomorphicity.
Given  two entire left slice monogenic functions $f$ and $g$,
  then their star-product (or slice hyperholomorphic product)
   is defined by
\begin{equation}
 (f\star_L g) (x) = \sum_{\ell=0}^{+\infty} x^\ell \sum_{k=0}^{\ell}a_{k}b_{\ell-k}.
\end{equation}
where $ f(x) = \sum_{k=0}^{+\infty}x^{k}a_k$ and $g(x) = \sum_{k=0}^{+\infty}x^kb_k$.

When we deal with monogenic functions the Fueter's polynomials
 $V_k(x)$ defined by
$$
V_k(x):= \frac{k!}{|k|!}\sum_{\sigma\in perm(k)} z_{j_{\sigma(1)}}z_{j_{\sigma(2)}} \ldots z_{j_{\sigma(|k|)}},
$$
where here $k$ is a multi-index,
play the same role as the monomials $x^k$, for $k\in \mathbb{N}_0$, of the paravector variable $x$ for slice monogenic functions.
The  CK-product  of two left entire monogenic $f$ and $g$ is defined by
$$
f \odot_L g:= \sum_{|k|=0}^{+\infty} \sum_{|j|=0}^{+\infty} V_{k+j}(x) f_k g_j,
$$
where
$
f(x)=\sum_{|k|=0}^\infty V_k(x)f_k,$ and $g(x)=\sum_{|k|=0}^\infty V_k(x)g_k$  are given
in terms of $V_k(x)$.

\medskip
(B) It is also possible to define slice hyperholomorphic functions, as functions in the kernel
  of the first order linear differential operator, introduced in \cite{global}, and defined by
 $$
\mathcal{G}f= \Big(|\underline{x}|^2\frac{\partial }{\partial x_0} +  \underline{x} \sum_{j=1}^n  x_j\frac{\partial }{\partial x_j}\Big)f=0,
$$
where $\underline{x}=x_1e_1+\ldots +x_ne_n$. The interesting observe
 that the operator $\mathcal{G}$ is linear nonconstant coefficients differential operator
 while the Dirac operator is linear but with constant coefficients.

\medskip
(C) The contour integral, in the Cauchy formula of slice monogenic functions and their derivatives, is computed on a
 the complex plane $C_{\mathbf{j}}$  in $\mathbb{R}^{n+1}$ (for $\mathbf{j}\in \mathbb{S}$).
 Such  contour is the boundary of $U\cap C_{\mathbf{j}}$, where the regular domain $U$ is contained
in $\mathbb{R}^{n+1}$ and is contained in a set where $f$ is slice monogenic.
For the monogenic case the integral, in the Cauchy formula
for monogenic functions and their derivatives, is computed on the boundary
of $U \subset \mathbb{R}^{n+1}$ where $\overline{U}$ is contained in the set of monogenicity of $f$.

\medskip
(D) For slice monogenic functions  there exists two different Cauchy kernels according to left and right slice hyperholomorphicity, while
left and right monogenic functions have same the Cauchy kernel.

\medskip
The main results are summarized as follows.
In the sections \ref{entireSLI} and \ref{entireMON} we collect the preliminary results on
function spaces of entire slice monogenic functions and of entire  monogenic functions with growth conditions, respectively.
 These results are of crucial importance in order to study
the continuity properties of a class of infinite order differential operators acting on entire slice monogenic and monogenic functions,
that are treated
in sections \ref{INSL} and \ref{INMON}, respectively.
We conclude this section with an overview of some of the main results.

\medskip
(I) Consider the formal infinite order differential operator
$$
U_L(x,\pp_{x_0})f(x):=\sum_{m=0}^\infty u_m(x)\star_L \pp_{x_0}^m f(x),
$$
defined on entire left slice monogenic functions $f$,
where $\star_L$ denotes the hyperholomorphic product.
Suppose that  $(u_m)_{m\in\mathbb{N}_0 }:\mathbb{R}^{n+1}\to \mathbb{R}_n$ is a sequence of entire left slice monogenic
 functions.
Assume  that $(u_m)_{m\in\mathbb{N}_0 }$ satisfy the condition such that there exists a constant $B>0$ so
 that for every $\varepsilon>0$ there exists $C_\varepsilon>0$ for which
 \begin{equation}
 |u_m(x)|\leq C_\varepsilon \frac{\varepsilon^m}{(m!)^{1/q}}\exp(B|x|^p), \ \ \ {\rm for \ all}  \ \ \ m\in \mathbb{N}_0,
 \end{equation}
where $1/p+1/q=1$ and $1/q=0$ when $p=1$.
Then in Theorem \ref{Mainslice} we show that
for $p\geq 1$ the  operator $U_L(x,\pp_{x_0})$
 acts continuously on the space of entire left slice monogenic functions with the growth condition
 $|f(x)|\leq C e^{B|x|^p}.$ In the same theorem we also considered right slice monogenic functions.

\medskip
(II)
For monogenic functions we  let $p\geq 1$ and set $\mathbb{N}_0=\mathbb{N} \cup\{0\}$.
Let $(u_m)_{m\in (\mathbb{N}_0)^n }:\mathbb{R}^{n+1}\to \mathbb{R}_n$ be  left entire monogenic functions
such that there exists a constant $B>0$ such
 that for every $\varepsilon>0$ there exists a constant $C_\varepsilon>0$ for which
\begin{equation}
 |u_m(x)|\leq C_\varepsilon \frac{\varepsilon^{|m|}}{(|m|!)^{1/q}}\exp(B|x|^p), \ \ \ {\rm for \ all}  \ \ \ m\in (\mathbb{N}_0)^n,
 \end{equation}
where $1/p+1/q=1$ and $1/q=0$ when $p=1$, and we observe that in this case $m$ is a multi-index.
We define the formal infinite order differential operator
$$
U_L(x,\pp_{x})f(x):=\sum_{|m|=0}^\infty u_m(x)\odot_L \pp_{x}^m f(x),
$$
for left entire monogenic functions $f$
where $\pp_{x}^m:= \pp_{x_1}^{m_1}\dots\pp_{x_n}^{m_n}$ and  $\odot_L$ denotes the CK-product.
Then for  $p\geq 1$, in Theorem \ref{t4}, we prove that the operator
$U_L(x,\pp_{x})$
 acts continuously on the space of left monogenic functions with the condition
 $|f(x)|\leq C e^{B|x|^p}.$

\medskip
Even though the two classes of hyperholomorphic functions have
very different Taylor series expansions
 they have strong similarities
with respect to the action of infinite order differential operators when
 we assume similar growth conditions on the coefficients of the operators.
 The results are even more surprising
 because of the exponential bounds $|f(x)|\leq C e^{B|x|^p}$ is used
for both classes of functions even though  the function $f(x)=e^{Bx}$, for $B\in \mathbb{R}$,
is slice monogenic but it is not monogenic.

\section{Function spaces of entire slice monogenic functions}\label{entireSLI}

In this section we recall some results on slice monogenic functions (see Chapter 2 in \cite{MR2752913})
and we prove some important properties of entire slice monogenic functions that appear here for the first time.
We recall that $\rr_n$ is the real Clifford algebra over $n$ imaginary units $e_1,\ldots ,e_n$.
The element
 $(x_0,x_1,\ldots,x_n)\in \mathbb{R}^{n+1}$ will be identified with the paravector
$
 x=x_0+\underline{x}=x_0+ \sum_{\ell=1}^nx_\ell e_\ell
$
and the real part $x_0$ of $x$ will also be denoted by ${\rm Re}(x)$.
An element in $\mathbb{R}_{n}$, called a {\em Clifford number}, can be written as
$$
a=a_0+a_1 e_1+\ldots +a_ne_n+a_{12}e_1e_2+\ldots+a_{123}e_1e_2e_3+\ldots+a_{12...n}e_1e_2...e_n.
$$
Denote by $A$  an element in the power set $P(1,\ldots ,n)$.
If $A=i_1\ldots i_r$, then the element $e_{i_1}\ldots e_{i_r}$ can be written as $e_{i_1...i_r}$ or, in short, $e_A$.
Thus, in a more compact form, we can write a Clifford number as
$$
a=\sum_Aa_A e_A.
$$
Possibly using the defining relations, we will order the indices in $A$ as $i_1 < \ldots <i_r$. When $A=\emptyset$ we set $e_\emptyset=1$.
The Euclidean norm of an element $y\in \mathbb{R}_n$ is
given by $|y|^2=\sum_{A} |y_A|^2$,
in particular the norm of the paravector $x\in\mathbb{R}^{n+1}$ is $|x|^2=x_0^2+x_1^2+\ldots +x_n^2$.
The conjugate of $x$ is given by
$
\bar x=x_0-\underline x=x_0- \sum_{\ell=1}^nx_\ell e_\ell.
$
Recall that $\mathbb{S}$ is the sphere
$$
\mathbb{S}=\{ \underline{x}=e_1x_1+\ldots +e_nx_n\ | \  x_1^2+\ldots +x_n^2=1\};
$$
so for $\mathbf{j}\in\mathbb{S}$ we have $\mathbf{j}^2=-1$.
Given an element $x=x_0+\underline{x}\in\rr^{n+1}$ let us define
$
\mathbf{j}_x=\underline{x}/|\underline{x}|$ if $\underline{x}\not=0,
$
 and given an element $x\in\rr^{n+1}$, the set
$$
[x]:=\{y\in\rr^{n+1}\ :\ y=x_0+{\mathbf{j}} |\underline{x}|, \ \mathbf{j}\in \mathbb{S}\}
$$
is an $(n-1)$-dimensional sphere in $\mathbb{R}^{n+1}$.
The vector space $\mathbb{R}+\mathbf{j}\mathbb{R}$ passing through $1$ and
$\mathbf{j}\in \mathbb{S}$ will be denoted by $\mathbb{C}_\mathbf{j}$ and
an element belonging to $\mathbb{C}_\mathbf{j}$ will be indicated by $u+\mathbf{j}v$, for $u$, $v\in \mathbb{R}$.
With an abuse of notation we will write $x\in\mathbb{R}^{n+1}$.
Thus, if $U\subseteq\mathbb{R}^{n+1}$ is an open set,
a function $f:\ U\subseteq \mathbb{R}^{n+1}\to\mathbb{R}_n$ can be interpreted as
a function of the paravector $x$.
We say that $U \subseteq \mathbb{R}^{n+1}$ is axially symmetric if $[x]\subset U$  for any $x \in U$.

\begin{definition}[Slice hyperholomorphic functions with values in $\mathbb{R}_n$ (or slice monogenic functions)]\label{SHolDefMON}
 Let $U\subseteq \mathbb{R}^{n+1}$ be an axially symmetric open set and let
 $\mathcal{U} = \{ (u,v)\in\rr^2: u+ \mathbb{S} v\subset U\}$. A function $f:U\to \mathbb{R}_n$ is called a left slice function, if it is of the form
 \[
 f(q) = f_{0}(u,v) + \mathbf{j}f_{1}(u,v)\qquad \text{for } q = u + \mathbf{j} v\in U
 \]
where the two functions $f_{0},f_{1}: \mathcal{U} \to \mathbb{R}_n$ satisfy the compatibility conditions
\begin{equation}\label{CCondmon}
f_{0}(u,-v) = f_{0}(u,v),\qquad f_{1}(u,-v) = -f_{1}(u,v).
\end{equation}
If in addition $f_{0}$ and $f_{1}$ satisfy the Cauchy-Riemann-equations
 \begin{align}\label{CRMMON}
\frac{\partial}{\partial u} f_{0}(u,v) - \frac{\partial}{\partial v} f_{1}(u,v) &= 0\\
\frac{\partial}{\partial v} f_{0}(u,v)+ \frac{\partial}{\partial u} f_{1}(u,v) &= 0,
\end{align}
 then $f$ is called left slice hyperholomorphic (or left slice monogenic).
A function $f:U\to \mathbb{R}_n$ is called a right slice function if it is of the form
\[
f(q) = f_{0}(u,v) + f_{1}(u,v) \mathbf{j}\qquad \text{for } q = u+ \mathbf{j}v \in U
\]
with two functions $f_{0},f_{1}: \mathcal{U} \to \mathbb{R}_n$ that satisfy \eqref{CCondmon}.
If in addition $f_{0}$ and $f_{1}$ satisfy the Cauchy-Riemann-equation, then $f$ is called right slice hyperholomorphic (or right slice monogenic).
\end{definition}

If $f$ is a left (or right) slice function such that $f_{0}$ and $f_{1}$ are real-valued, then $f$ is called intrinsic.
We denote the sets of left and right slice hyperholomorphic functions on $U$ by $\mathcal{S\!M}_L(U)$
and $\mathcal{S\!M}_R(U)$, respectively. When we do not distinguish between left of right we indicate the space  $\mathcal{S\!M}(U)$.

\begin{definition}\label{slicederivative}
Let $f: U\subseteq\mathbb{R}^{n+1}\to\mathbb{R}_n$ and let $x = u + \mathbf{j}v\in U$. If $x$ is not real, then we say that $f$ admits left slice derivative in $x$ if
\begin{equation}\label{SDerivL}
\partial_S f(x) := \lim_{p\to x, \,  p\in\cc_{\mathbf{j}}} (p-x)^{-1}(f_\mathbf{j}(p)-f_\mathbf{j}(x))
\end{equation}
exists and is finite. If $x$ is real, then we say that $f$ admits left slice derivative in $x$ if \eqref{SDerivL} exists for any $\mathbf{j}\in\mathbb{S}$.
Similarly, we say that $f$ admits right slice derivative in a nonreal point $x = u + \mathbf{j}v \in U$ if
\begin{equation}\label{SDerivR}
\partial_S f(x) := \lim_{p\to x, \,  p\in\cc_{\mathbf{j}}}(f_\mathbf{j}(p)-f_\mathbf{j}(x))(p-x)^{-1}
\end{equation}
exists and is finite, and we say that  $f$ admits right slice derivative in a real point
$x\in U$ if \eqref{SDerivR} exists and is finite, for any $\mathbf{j}\in\mathbb{S}$.
\end{definition}
\begin{remark}
Observe that $\partial_S f(x)$ is uniquely defined and independent of the choice of $\mathbf{j}\in\mathbb{S}$ even if $x$ is real.
If $f$ admits slice derivative, then $f_{\mathbf{j}}$ is $\cc_\mathbf{j}$-complex left resp. right differentiable and we find
\begin{equation}\label{SDerivPartial}
\partial_S f(x) = f_{\mathbf{j}}'(x) = \frac{\partial}{\partial u} f_{\mathbf{j}}(x) = \frac{\partial}{\partial u}f(x),\quad x = u + \mathbf{j}v.
\end{equation}
\end{remark}

\begin{theorem}\label{PowSerThm}
Let $a\in\mathbb{R}$, let $r>0$ and let $B_{r}(a) = \{x\in\hh: |x-a|<r\}$. If $f\in\mathcal{S\!M}_L(B_r(a))$, then
\begin{equation}
\label{PowSerL}
f(x) = \sum_{k= 0}^{+\infty} (x-a)^k\frac{1}{k!} \partial_S^k f(a)\qquad \forall x = u + \mathbf{j} v \in B_r(a).
\end{equation}
If on the other hand $f\in\mathcal{S\!M}_R(B_r(a))$, then
\[
f(x) = \sum_{k= 0}^{+\infty}\frac{1}{k!}\left(\partial_S^k f(a)\right)  (x-a)^k\qquad \forall x = u + \mathbf{j} v \in B_r(a).
\]
\end{theorem}
We now recall the natural product that preserves slice monogenicity of functions admitting power series expansion
as shown by Theorem \ref{PowSerThm}.
\begin{definition}
Let  $f(x) = \sum_{k=0}^{+\infty}x^{k}a_k$ and
$g(x) = \sum_{k=0}^{+\infty}x^kb_k$ be two left slice monogenic power series,
the left-star product, denoted by $\star_L$, is defined by
\begin{equation}\label{ProdLSeries}
  (f\star_L g) (x) = \sum_{\ell=0}^{+\infty} x^\ell \left(\sum_{k=0}^{\ell}a_{k}b_{\ell-k}\right).
\end{equation}
Similarly, for right slice monogenic power series
$
f(x)=\sum_{k=0}^{+\infty}a_kx^{k}
$ and
$
g(x)=\sum_{k=0}^{+\infty}b_kx^{k}
$
the right-star product, denoted by $\star_R$,
is defined by
\begin{equation}\label{ProdRSeries}
 (f\star_R g) (x) =  \sum_{\ell=0}^{+\infty} \left( \sum_{k=0}^{\ell}a_{k}b_{\ell-k}\right)x^\ell.
\end{equation}
\end{definition}
The Cauchy formula of slice monogenic functions has two different Cauchy kernels
according to left or right slice monogenicity.
Let $x,s\in \mathbb{R}^{n+1}$, with $x\not\in [s]$,
be paravectors then the slice monogenic Cauchy kernels are defined by
$$
S_L^{-1}(s,x):=-(x^2 -2 {\rm Re}  (s) x+|s|^2)^{-1}(x-\overline s),
$$
and
$$
S_R^{-1}(s,x):=-(x-\bar s)(x^2-2{\rm Re} (s)x+|s|^2)^{-1}.
$$
\begin{theorem}[The Cauchy formulas for slice monogenic functions]
\label{CauchygeneraleMONOG}
Let $U\subset\mathbb{R}^{n+1}$ be an axially symmetric domain.
Suppose that $\partial (U\cap \mathbb{C}_\mathbf{j})$ is a finite union of
continuously differentiable Jordan curves  for every $\mathbf{j}\in\mathbb{S}$ and set
 $ds_\mathbf{j}=-ds \mathbf{j}$ for $\mathbf{j}\in \mathbb{S}$.
 Let $f$ be
a slice monogenic function on an open set that contains $\overline{U}$ and set
$x= x_0+\underline{x}$, $s=s_0+\underline{s}$.
 Then
\begin{equation}\label{cauchynuovo}
 f(x)=\frac{1}{2 \pi}\int_{\partial (U\cap \mathbb{C}_\mathbf{j})} S_L^{-1}(s,x)\, ds_\mathbf{j}\,  f(s),\qquad\text{for any }\ \  x\in U.
\end{equation}
If $f$ is a right slice monogenic function on a set that contains $\overline{U}$,
then
\begin{equation}\label{Cauchyright}
 f(x)=\frac{1}{2 \pi}\int_{\partial (U\cap \mathbb{C}_\mathbf{j})}  f(s)\, ds_\mathbf{j}\, S_R^{-1}(s,x),\qquad\text{for any }\ \  x\in U.
 \end{equation}
 Moreover, the integrals  depend neither on $U$ nor on the imaginary unit $\mathbf{j}\in\mathbb{S}$.
\end{theorem}

\begin{theorem}[Derivatives of slice monogenic functions]
Let $U\subset\mathbb{R}^{n+1}$ be an axially symmetric domain.
Suppose that $\partial (U\cap \mathbb{C}_\mathbf{j})$ is a finite union of
continuously differentiable Jordan curves  for every $\mathbf{j}\in\mathbb{S}$ and set
 $ds_\mathbf{j}=-ds \mathbf{j}$ for $\mathbf{j}\in \mathbb{S}$.
 Let $f$ be
a left slice monogenic function on an open set that contains $\overline{U}$ and set
$x= x_0+\underline{x}$, $s=s_0+\underline{s}$. Then the slice derivatives $\partial^k_{x_0} f(x)$ are given by
 \begin{equation}\label{sderiv}
\partial^k_{x_0} f(x)
=\frac{k!}{2 \pi} \int_{\partial (U\cap \mathbb{C}_\mathbf{j})}
(x^2-2s_0x+|s|^2)^{-k-1} (x-\overline{s})^{*(k+1)} ds_\mathbf{j} f(s),
\end{equation}
where
\begin{equation}\label{stellina}
(x-\overline{s})^{*k}=\sum_{m=0}^{k}\frac{k!}{(k-m)!m!}
x^{k-m}\overline{s}^m,
\end{equation}
 Moreover,
the integral  depends neither on $U$ nor on the imaginary unit $\mathbf{j}\in\mathbb{S}$.
\end{theorem}
A similar formula holds also for right slice monogenic functions.

\medskip
After the basic facts on slice monogenic functions we can introduce some function spaces of entire
slice monogenic functions in the spirit or the quaternionic version introduced in the book \cite{ACSBOOK2}.
Let $f$  be a non-constant entire monogenic function. We define
$$
M_{f_\mathbf{j}}(r)=\max_{|z|=r, x\in \mathbb{C}_\mathbf{j} }|f(x)|,\ \ \ \text{ for}\ \ r\geq 0
$$
and
$$
M_f(r)=\max_{|x|=r}|f(x)|,\ \ \ \text{ for}\ \ r\geq 0.
$$
Then, see Chapter 5 \cite{ACSBOOK2}, we have for intrinsic functions that $M_{f_\mathbf{j}}(r)=M_f(r)$.
\begin{definition}
Let $f$ be an entire slice monogenic function. Then we say that $f$ is of finite order if there exists $\kappa>0$ such that
$$
M_f(r)< e^{r^\kappa}
$$
for sufficiently large $r$. The greatest lower bound $\rho$ of such numbers $\kappa$ is called order of $f$.
Equivalently, we can define the order as
$$
\rho=\limsup_{r\to\infty}\frac{\ln\ln M_f(r)}{\ln r}.
$$
\end{definition}
\begin{definition}
Let $f$  be an entire slice monogenic function of order $\rho$ and let $A>0$ be such that for sufficiently large values of $r$ we have
$$
M_f(r)< e^{Ar^\rho}.
$$
We say that $f$ of order $\rho$ is of type $\sigma$ if $\sigma$ is the greatest lower bound of such numbers and we have
$$
\sigma=\limsup_{r\to\infty}\frac{\ln M_f(r)}{ r^\rho}.
$$
 Moreover:
\begin{itemize}
\item
When $\sigma=0$ we say that $f$ is of minimal type.
\item
When $\sigma=\infty$ we say that $f$ is of maximal type.
\item
When $\sigma \in (0,\infty)$ we say that $f$ is of normal type.
\end{itemize}
\end{definition}
The constant functions are said to be of minimal type of order zero.

\begin{definition} Let $p\geq 1$.
We denote by $\mathcal{S\!M}^{p}$ the space of entire slice monogenic
functions with either order lower than $p$ or order equal to $p$ and finite type.
It consists of those functions $f:\mathbb{R}^{n+1}\to \mathbb{R}_n$, for which there exist constants $B, C >0$ such that
\begin{equation}\label{ABC}
|f(x)|\leq C e^{B|x|^p}.
\end{equation}
Let $(f_m)_{m\in \mathbb{N} }$, $f_0\in \mathcal{S\!M}^p$.
Then $f_m \to f_0$  in $\mathcal{S\!M}^p$ if there exists some $B > 0$ such that
\begin{equation}
\lim\limits_{m\rightarrow\infty} \sup_{x\in \mathbb{R}^{n+1}}\Big|(f_m(x)-f_0(x))e^{-B|x|^p}\Big|=0.
\end{equation}
Functions in $\mathcal{S\!M}^p$ that are left slice monogenic will be denoted by $\mathcal{S\!M}^p_L$, while
right slice monogenic will be denoted by $\mathcal{S\!M}^p_R$.
\end{definition}
We now give a characterization of functions in $\mathcal{S\!M}^p$ in terms of their Taylor coefficients.
In order to prove our results we need
 some very well known estimates on the Gamma function $\Gamma$ and on the binomial. We collect them
in the following lemma.
\begin{lemma}\label{estigam}
We have the following estimates:
\begin{itemize}
\item[(I)]
 For $j$, $k\in \mathbb{N}$, the we have
$
(j+k)!\leq 2^{j+k}j!k!.
$
\item[(II)]
 For $n$, $k\in \mathbb{N}$, the we have
$
\Gamma(n+1)\Gamma(k+1)\leq \Gamma(n+k+2).
$
\item[(III)]
 For $q\in [1,\infty)$ and $n\in \mathbb{N}$ we have
$
\Gamma\Big(\frac{n}{q}+1\Big)\leq (n!)^{1/q}.
$
\item[(IV)]
$
(a+b)^p\leq 2^p(a^p+b^p),\ \ a>0, \ \ b>0, \ \ p>0.
$
\end{itemize}
\end{lemma}
\begin{lemma} Let $x\in \mathbb{R}^{n+1}$ then the
Mittag-Leffler function
$$
E_{\alpha,\beta}(x)=\sum_{k=1}^\infty\frac{x^k}{\Gamma(\alpha k+\beta)}
$$
is an entire slice monogenic function of order $1/\alpha$  (and of type 1) for $\alpha >0$ and $Re(\beta)>0$.
\end{lemma}
\begin{proof}
The proof follows the same line as in  complex case.
\end{proof}
We are now ready to prove a crucial result which is the slice monogenic version of the complex version proved in \cite{ACSS18}.
\begin{lemma}\label{CoefLemma} Let $p\geq1$.
A function
$$
f(x)=\sum_{k=0}^\infty  x^k\alpha_k
$$
 belongs to $\mathcal{S\!M}_L^p$ if and only if there exist constants $C_f,b_f>0$ such that
\begin{equation}\label{estfj}
|\alpha_k|\leq C_f \frac{b_f^k}{\Gamma(\frac{k}{p}+1)}.
\end{equation}
Furthermore, a sequence $f_m$ in $\mathcal{S\!M}_L^p$ tends to zero if and only if $C_{f_m}\to0$ and $b_{f_m} < b$ for some $b>0$.
 \end{lemma}
\begin{proof} We consider the case of left slice monogenic functions the prove for right slice monogenic functions follows the same lines.
We first prove that if $f\in \mathcal{S\!M}_L^p$ we have the estimates (\ref{estfj}) on the coefficients $\alpha_k$, for $k\in \mathbb{N}_0$.
Observe that the kernel
$$
(x,s)\mapsto (x^2-2s_0x+|s|^2)^{-k-1} (x-\overline{s})^{*(k+1)}
$$
in formula (\ref{sderiv})
can be decomposed by the Representation formula.
Moreover,
the zeros of the function $x\mapsto x^2-2s_0x+|s|^2$  consist of real points or of a 2-sphere.
In fact, on $\mathbb{C}_{\mathbf{j}_x}$ we find only the point $x$ as a singularity and the result follows from the Cauchy formula on the plane $\mathbb{C}_{\mathbf{j}_x}$.
In  the complex plane $\mathbb{C}_\mathbf{j}$ for $\mathbf{j}\not=\mathbf{j}_x$ if the singularities are real we obtain again the Cauchy formula of complex analysis.
It the zeros are  not real and $\mathbf{j}\not=\mathbf{j}_x$ then on any complex plane
$C_\mathbf{j}$ we find the two zeros $s_{1,2}=x_0\pm \mathbf{j}|\underline{x}|$ in this case using  the representation formula we
have the decomposition:
$$
(x^2-2s_0x+|s|^2)^{-k-1} (x-\overline{s})^{*(k+1)}=
\frac{1-\mathbf{ij}}{2}\ \frac{1}{(s-w)^{k+1}}+\frac{1+\mathbf{ij}}{2}\ \frac{1}{(s-\overline{w})^{k+1}}
$$
for $x=u+\mathbf{i}v$ and $w=u+\mathbf{j}v$.
So the integral representation of the derivatives becomes
$$
\partial^k_{x_0} f(x)
=\frac{k!}{2 \pi} \int_{\partial (U\cap \mathbb{C}_\mathbf{j})}\Big(
\frac{1-\mathbf{ij}}{2}\ \frac{1}{(s-w)^{k+1}}+\frac{1+\mathbf{ij}}{2}\ \frac{1}{(s-\overline{w})^{k+1}}\Big)ds_\mathbf{j} f(s)
$$
and also
\begin{equation}\label{dervate}
\begin{split}\partial^k_{x_0} f(x)
&= \frac{1-\mathbf{ij}}{2}\ \frac{k!}{2 \pi} \ \int_{\partial (U_1\cap \mathbb{C}_\mathbf{j})}
\frac{1}{(s-w)^{k+1}}ds_\mathbf{j} f(s)
\\
&
+\frac{1+\mathbf{ij}}{2}\ \frac{k!}{2 \pi} \int_{\partial (U_2\cap \mathbb{C}_\mathbf{j})}\frac{1}{(s-\overline{w})^{k+1}}ds_\mathbf{j} f(s),
\end{split}
\end{equation}
where $\partial (U_1\cap \mathbb{C}_\mathbf{j})$ is the path of integration
in $\mathbb{C}_\mathbf{j}$ that contains the point  $s_{1}=x_0+\mathbf{j}|\underline{x}|$
and
$\partial (U_2\cap \mathbb{C}_\mathbf{j})$
is the path of integration
in $\mathbb{C}_\mathbf{j}$ that contains the point  $s_{2}=x_0-\mathbf{j}|\underline{x}|$.
Now we suppose that the above paths of integration are the two circles
$|s-s_1|=\tau|s_1|$ and $|s-s_2|=\tau|s_2|$  where $\tau>0$ is a parameter.
Now, we estimate the two terms
$$
f^{(k)}(w):= \frac{k!}{2 \pi} \ \int_{\partial (U_1\cap \mathbb{C}_\mathbf{j})}
\frac{1}{(s-w)^{k+1}}ds_\mathbf{j} f(s)
$$
and
$$
f^{(k)}(\overline{w}):=\frac{k!}{2 \pi} \int_{\partial (U_2\cap \mathbb{C}_\mathbf{j})}\frac{1}{(s-\overline{w})^{k+1}}ds_\mathbf{j} f(s)
$$
using the Cauchy formula for the derivatives in the complex plane  $\mathbb{C}_\mathbf{j}$.
Recalling that we assume the growth condition $|f(x)|\leq C_f e^{B|x|^p}$, we obtain:
\[
\begin{split}
|f^{(k)}(w)|&\leq \frac{k!}{(\tau|w|)^j} \max_{|w-z|=\tau|z|}|f(w)|
\leq \frac{C_fk!}{(\tau|w|)^k}\exp(B(1+\tau)^p|w|^p)
\end{split}
\]
and similarly for $f^{(k)}(\overline{w})$
\[
\begin{split}
|f^{(k)}(\overline{w})|
\leq \frac{C_fk!}{(\tau|\overline{w}|)^k}\exp(B(1+\tau)^p|\overline{w}|^p)
\end{split}
\]
so we conclude that
\[
\begin{split}
|f^{(k)}(w)|+|f^{(k)}(\overline{w})|
\leq 2\frac{C_fk!}{(\tau|w|)^k}\exp(B(1+\tau)^p|w|^p)
\end{split}
\]
for all $\tau> 0$, where we have used the fact that $f\in A_p$ and $|w|\leq (1+s)|z|$.
Now we can estimate the slice derivative from the formula (\ref{dervate}), precisely
$$
\partial^k_{x_0} f(x)= \frac{1-\mathbf{ij}}{2}\ f^{(k)}(w)+\frac{1+\mathbf{ij}}{2}\ f^{(k)}(\overline{w})
$$
gives
$$
|\partial^k_{x_0} f(x)|\leq |\frac{1-\mathbf{ij}}{2}|\ |f^{(k)}(w)|+|\frac{1+\mathbf{ij}}{2}|\ |f^{(k)}(\overline{w})|\leq |f^{(k)}({w})|+|f^{(k)}(\overline{w})|.
$$
The well known estimate (IV) in Lemma \ref{estigam}
gives $ (1+\tau)^p\leq 2^p(\tau^p+1)$ for all $\tau>0$.
Hence we have
\begin{equation}\label{stmadamin}
|f^{(k)}(w)|+|f^{(k)}(\overline{w})|
\leq 2C_f\frac{k!}{(\tau|w|)^k}\exp(B\cdot 2^p\tau^p|w|^p)\exp(B\cdot 2^p|w|^p)
\end{equation}
for all $w\in \mathbb{C}_\mathbf{j}$ and $\tau>0$.
Now we observe that the point
$\tau_{\rm min}=\Big(\frac{k}{2^pBp}\Big)^{1/p}\frac{1}{|w|}$
 for $w\not=0$ is the minimum
of the function
$$
\tau\mapsto \frac{1}{(\tau|w|)^k}\exp(B\cdot 2^p\tau^p|w|^p)
$$
which is the right-hand side of (\ref{stmadamin}), so that we obtain
$$
|\partial^k_{x_0} f(x)|\leq|f^{(k)}(w)|+|f^{(k)}(\overline{w})|
\leq 2C_f\  k! \Big(\frac{2^pBp }{k}\Big)^{k/p}  e^{k/p}\exp(A2^p|w|^p).
$$
If we set
$$
b:=(2^pBp e)^{1/p}
$$
we deduce the estimate
$$
|\partial^k_{x_0} f(x)|\leq|f^{(k)}(w)|+|f^{(k)}(\overline{w})|\leq 2C_f k!\frac{b^k}{k^{k/p}}   \exp(B\cdot 2^p|w|^p)
$$
for all $w\in \mathbb{C}_\mathbf{j}$.
Since
$$
\alpha_k=\frac{\partial^k_{x_0}(0)}{k!}
$$
we have, by the maximum modules principle applied in a disc centered at
the origin and with radius $\epsilon>0$ sufficiently small, in the complex plane
$\mathbb{C}_\mathbf{j}$
\[
\begin{split}
|\alpha_k|\leq  C_f \ \frac{b^k}{k^{k/p}}  \exp(B\cdot 2^p\epsilon^p)
\leq  2C_f \ \frac{b^k}{k^{k/p}}
= C'_f \ \frac{b^k}{(k!)^{1/p}}
\leq C'_f \ \frac{b^k}{\Gamma(\frac{k}{p}+1)}.
\end{split}
\]
The other direction follows form the properties of the Mittag-Leffler function because it is
of order $1/\alpha$  (and of type 1) for $\alpha >0$ and $Re(\beta)>0$,
so, in our case, $f$ is entire of order $p$.
The fact that
 $f_m$ in $\mathcal{S\!M}^p$ tends to zero if and only if $C_{f_m}\to0$ and $b_{f_m} < b$ for some $b>0$ is a consequence of the estimate
 on the $\alpha_k$.

\end{proof}

\section{Infinite order differential operators on slice monogenic functions}\label{INSL}

In this section we study a class of infinite order differential operators acting on spaces of entire slice monogenic functions.
The definition of these infinite order differential operators
preserve the slice monogenicity. In fact,
we consider the star-product of the coefficients $(u_m)_{m\in\mathbb{N}_0 }$, where $
 \mathbb{N}_0=\mathbb{N} \cup\{0\}$, of the operator and the
slice derivative $\pp_{x_0}^m f(x)$ of the slice monogenic function $f:\mathbb{R}^{n+1}\to \mathbb{R}_n$. Precisely we have:
\begin{definition}\label{OPRSLICE}
Let $p\geq 1$.
\begin{itemize}
\item
Let $(u_m)_{m\in\mathbb{N}_0 }:\mathbb{R}^{n+1}\to \mathbb{R}_n$ be  entire functions in $\mathcal{S\!M}_L$.
We define the set $\mathbf{D}^L_{p,0}$ of formal
operators defined by
$$
U_L(x,\pp_{x_0})f(x):=\sum_{m=0}^\infty u_m(x)\star_L \pp_{x_0}^m f(x),
$$
for entire functions $f$ in $\mathcal{S\!M}_L$.
\item
Let $(u_m)_{m\in \mathbb{N}_0}:\mathbb{R}^{n+1}\to \mathbb{R}_n$ be  entire functions in $\mathcal{S\!M}_R$.
We define the set $\mathbf{D}^R_{p,0}$ of formal
operators defined by
$$
U_R(x,\pp_{x_0})f(x):=\sum_{m=0}^\infty u_m(x)\star_R \pp_{x_0}^m f(x)
$$
for entire functions $f$ in $\mathcal{S\!M}_R$.
\end{itemize}
The entire functions $(u_m)_{m\in \mathbb{N}_0}$ in $\mathcal{S\!M}_L$ (resp. in $\mathcal{S\!M}_R$) satisfy the additional condition: There exists a constant $B>0$ such
 that for every $\varepsilon>0$ there exists a constant $C_\varepsilon>0$ for which
 \begin{equation}\label{estun}
 |u_m(x)|\leq C_\varepsilon \frac{\varepsilon^m}{(m!)^{1/q}}\exp(B|x|^p), \ \ \ {\rm for \ all}  \ \ \ m\in \mathbb{N}_0,
 \end{equation}
where $1/p+1/q=1$ and $1/q=0$ when $p=1$.
\end{definition}

We are now in the position to state and proof the main result of this section.
\begin{theorem}\label{Mainslice}
Let $p\geq 1$ and let $\mathbf{D}^L_{p,0}$ and $\mathbf{D}^R_{p,0}$ be the sets of formal operators in Definition \ref{OPRSLICE}.
\begin{itemize}
\item[(I)]
Let  $U_L(x,\pp_{x_0})\in \mathbf{D}^L_{p,0}$ and let $f\in \mathcal{S\!M}^p_L$, then $U_L(x,\pp_{x_0})f\in \mathcal{S\!M}^p_L$ and
the operator $U_L(x,\pp_{x_0})$ acts continuously on $\mathcal{S\!M}^p_L$, i.e., if $f_m\in \mathcal{S\!M}^p_L$
and  $f_m\to 0$ in $\mathcal{S\!M}^p_L$ then $U_L(x,\pp_{x_0})f_m\to 0$ in $\mathcal{S\!M}^p_L$.
\item[(II)]
Let  $U_L(x,\pp_{x_0})\in \mathbf{D}^R_{p,0}$ and let $f\in \mathcal{S\!M}^p_R$, then $U_R(x,\pp_{x_0})f\in \mathcal{S\!M}^p_R$ and
the operator $U_R(x,\pp_{x_0})$ acts continuously on $\mathcal{S\!M}^p_R$, i.e., if $f_m\in \mathcal{S\!M}^p_R$
and  $f_m\to 0$ in $\mathcal{S\!M}^p_R$ then $U_R(x,\pp_{x_0})f_m\to 0$ in $\mathcal{S\!M}^p_R$.
\end{itemize}
\end{theorem}
\begin{proof}
Let us prove case (I), since case (II) follows with similar computations.
We apply operator $U_L(x,\pp_{x_0})\in \mathbf{D}^L_{p,0}$ (see Definition \ref{OPRSLICE}) to a function  $f\in \mathcal{S\!M}^p_L$,
\[
\begin{split}
U_L(x,\pp_{x_0})f(x)&=\sum_{m=0}^\infty u_m(x)\star_L \pp_{x_0}^m \sum_{j=0}^\infty \alpha_jx^j
\\
&
=\sum_{m=0}^\infty u_m(x)\star_L \sum_{j=0}^\infty \alpha_j\pp_{x_0}^m x^j
\\
&
=\sum_{m=0}^\infty u_m(x)\star_L \sum_{j=m}^\infty \alpha_j\frac{j!}{(j-m)!} x^{j-m}
\\
&
=\sum_{m=0}^\infty\sum_{k=0}^\infty  \alpha_{m+k} u_m(x) \frac{(k+m)!}{k!} x^{k}.
\end{split}
\]
Now we observe that
\[
\begin{split}
|U_L(x,\pp_{x_0})f(x)|&\leq
\sum_{m=0}^\infty\sum_{k=0}^\infty  |\alpha_{m+k}| |u_m(x)| \frac{(k+m)!}{k!} |x|^{k}
\end{split}
\]
and we recall that since  $U_L(x,\pp_{x_0})\in \mathbf{D}^L_{p,0}$ the coefficients $u_m(x)$ of the operator satisfy estimate
(\ref{estun}) and since $f\in \mathcal{S\!M}^p_L$, the coefficients $|\alpha_{k}|$ of $f$ satisfy estimate (\ref{estfj}) so we have
\[
\begin{split}
|U_L(x,\pp_{x_0})f(x)|&\leq
C_f C_\varepsilon\sum_{m=0}^\infty  \sum_{k=0}^\infty   \frac{\varepsilon^m}{(m!)^{1/q}} \exp(B |x|^p)\ \
 \frac{b^{m+k}}{\Gamma\Big(\frac{m+k}{p}+1\Big)} \frac{(k+m)!}{k!} |x|^{k}.
\end{split}
\]
We now use estimates (I) and (III) in Lemma \ref{estigam}
to get the estimates
\begin{equation}\label{ineq}
\begin{split}
|U_L(x,\pp_{x_0})f(x)|&\leq
C_f C_\varepsilon\sum_{m=0}^\infty  \sum_{k=0}^\infty   \frac{\varepsilon^m}{\Gamma\Big(\frac{m}{q}+1\Big)} \ \
 \frac{b^{m+k}}{\Gamma\Big(\frac{m+k}{p}+1\Big)} \frac{2^{k+m}k!m! }{k!} |x|^{k}\exp(B |x|^p)
 \\
 &
 \leq
C_f C_\varepsilon\sum_{m=0}^\infty  \sum_{k=0}^\infty  (2b)^k (2\varepsilon b)^m
\frac{1}{\Gamma\Big(\frac{m}{q}+1\Big)} \ \
 \frac{m!}{\Gamma\Big(\frac{m+k}{p}+1\Big)}   |x|^{k}\exp(B |x|^p),
\end{split}
\end{equation}
from (III) in Lemma \ref{estigam}
it follows that
$
\Gamma\Big(\frac{k+m}{p}+1\Big)\geq \Gamma\Big(\frac{k}{p}+\frac{1}{2}\Big)\Gamma\Big(\frac{m}{p}+\frac{1}{2}\Big)
$
and so we can write \eqref{ineq} as
\[
\begin{split}
|U_L(x,\pp_{x_0})f(x)|& \leq \beta(p,q,b,\varepsilon)
C_f C_\varepsilon\sum_{k=0}^\infty   \frac{(2b)^k}{\Gamma\Big(\frac{k}{p}+\frac{1}{2}\Big)} |x|^{k}\exp(B |x|^p)
\end{split}
\]
with the position
\begin{equation}\label{betaserie}
\beta(p,q,b,\varepsilon):=\sum_{m=0}^\infty  (2\varepsilon b)^m
 \frac{m!}{\Gamma\Big(\frac{m}{p}+\frac{1}{2}\Big)\Gamma\Big(\frac{m}{q}+1\Big)}
\end{equation}
where we can show that the series (\ref{betaserie})
is convergent, for $\varepsilon $ is arbitrary small the series converges, using the asymptotic expansion of the Gamma function
 $$
  (2\varepsilon b)^m
 \frac{m!}{\Gamma\Big(\frac{m}{p}+\frac{1}{2}\Big)\Gamma\Big(\frac{m}{q}+1\Big)}  \sim\frac{m^m(2\varepsilon b)^m   }{\Big(\frac{m}{p}\Big)^{m/p}\Big(\frac{m}{q}\Big)^{m/q} }=(2\varepsilon b)^m [p^{1/p}q^{1/q}]^m.
 $$
We finally obtain
 \[
\begin{split}
|U_L(x,\pp_{x_0})f(x)|& \leq
\beta(p,q,b,\varepsilon)C_f C_\varepsilon\sum_{k=0}^\infty   \frac{(2b)^k}{\Gamma\Big(\frac{k}{p}+\frac{1}{2}\Big)} |x|^{k}\exp(B |x|^p)
\end{split}
\]
and, by the properties of the Mittag-Leffler function, we have
$$
\sum_{k=0}^\infty   \frac{(2b)^k}{\Gamma\Big(\frac{k}{p}+\frac{1}{2}\Big)} |x|^{k}\leq C'\exp(B' |x|^p).
$$
We conclude that
 there exists $B''>0$ such that
\[
\begin{split}
|U_L(x,\pp_{x_0})f(x)|& \leq
\beta(p,q,b,\varepsilon)C_f C_\varepsilon\exp(B'' |x|^p)
\end{split}
\]
 that is $U_L(x,\pp_{x_0})f(x)\in \mathcal{S\!M}^p_L$ and for $C_{f_m}\to 0$  the same estimate proves the continuity,
 i.e. $|U_L(x,\pp_{x_0})f_m(x)|\to 0$ when $f_m(x)\to 0$.
\end{proof}


\section{Function spaces of entire monogenic functions}\label{entireMON}

We recall the in the sequel we work in the real Clifford algebra $\mathbb{R}_n$,
so for the definition and the notation we refer the reader to Section \ref{entireSLI}, for more
 more details on monogenic functions see the book of \cite{BDS82}.
 We start with some definitions. The generalized Cauchy-Riemann operator in $\mathbb R^{n+1}$ is defined by:
$$
\mathcal D:= \frac{\partial}{\partial_{x_0}}+\sum_{i=1}^n e_i \frac{\partial}{\partial_{x_i}}.
$$

\begin{definition}[Left and right Monogenic Functions] Let $U\subseteq\mathbb R^{n+1}$ be an open subset.
 A function $f:\, U\to \mathbb{R}_n$, of class $C^1$, is called left monogenic if
$$
\mathcal D f= \frac{\partial}{\partial_{x_0}} f +\sum_{i=1}^n e_i \frac{\partial}{\partial_{x_i}} f=0.
$$
A function $g:\, U\to \mathbb{R}_n$, of class $C^1$, is called right monogenic if
$$
f \mathcal D = \frac{\partial}{\partial_{x_0}} f +\sum_{i=1}^n \frac{\partial}{\partial_{x_i}} f e_i=0.
$$
The set of  left monogenic functions (resp. right monogenic functions) will be denoted by $\mathcal M_L(U)$
(resp. $\mathcal M_R(U)$); if $U=\mathbb R^{n+1}$ we simply denote it by $\mathcal M_L$ (resp. $\mathcal M_R$)
\end{definition}
\begin{definition}[Fueter's homogeneous polynomials]\label{fp}
Given a multi--index $k=(k_1,...,k_n)$ where $k_i\geq 0$, we set $|k|=\sum_{i=1}^nk_i$ and $k!=\prod_{i=1}^nk_i!$.

\medskip
(I) For a multi--index $k$ with at least one negative component we set
$$
P_k(x):=0
$$
for $0=(0,...,0)$ we set
$$
P_0(x):=1.
$$

\medskip
(II)
For a multi-index $k$ with $|k|>0$ we define $P_k(x)$ as follows: for each $k$ consider the sequence of indices $j_1,j_2,\ldots ,j_{|k|}$ be given such that the first $k_1$ indices equals $1$, the next indices $k_2$ equals $2$ and, finally, the last $k_n$ equals $n$.
We define $z_i=x_ie_0-x_0e_i$ for any $i=1,\dots , n$ and $z=(z_1,\dots, z_n)$. We set
$$
z^k:=z_{j_1}z_{j_2}\ldots z_{j_{|k|}}=z_1^{k_1}z_2^{k_2}\ldots z_n^{k_n}
$$
and
$$
|z|^k=|z_1|^{k_1}\cdots |z_n|^{k_n}
$$
these products contains $z_1$ exactly $k_1$-times, $z_2$ exactly $k_2$-times and so on. We define
$$
P_k(x)=\frac{1}{|k|!}\sum_{\sigma\in perm(k)}\sigma(z^k):=\frac{1}{|k|!}\sum_{\sigma\in perm(k)} z_{j_{\sigma(1)}}z_{j_{\sigma(2)}} \ldots z_{j_{\sigma(|k|)}},
$$
where $perm(k)$ is the permutation group with $|k|$ elements. When we multiply by $k!$ the Fueter's polynomial $P_k(x)$ we will denote it by $V_k(x)$ i.e.
$$ V_k(x):= k! P_k(x)= \frac{k!}{|k|!}\sum_{\sigma\in perm(k)} z_{j_{\sigma(1)}}z_{j_{\sigma(2)}} \ldots z_{j_{\sigma(|k|)}} $$
\end{definition}

These polynomials play an important role in the monogenic function theory and we collect some of their properties in the next proposition (see Theorem $6.2$ in \cite{GHS08}):
\begin{theorem}\label{fpt}
Consider the Fueter polynomials $P_k(x)$ defined above. Then the following facts hold:.

(I) the recursion formula
$$
kP_k(x)=\sum_{i=1}^mk_i P_{k-\varepsilon_i}(x)z_i=\sum_{i=1}^mk_i z_iP_{k-\varepsilon_i}(x),
$$
and also
$$
\sum_{i=1}^mk_i P_{k-\varepsilon_i}(x)e_i=\sum_{i=1}^mk_i e_iP_{k-\varepsilon_i}(x),
$$
where $\varepsilon_i=(0,...,0,1,0,...,0)$ with $1$ in the position $i$.

(II) The derivatives $\partial_{x_j}$ for $j=1,...,n$, are given by
$$
\partial_{x_j}P_k(x)=k_jP_{k-\varepsilon_j}(x).
$$

(III) The Fueter Polynomials $P_k(x)$ are both left and right monogenic.

(IV) The following estimates holds
$$
|P_k(x)|\leq |x|^{|k|}.
$$

(V) (Binomial formula)
For all paravectors $x$ and $y$, and for the multi--index $k$, $j$ and $i$
$$
P_k(x+y)=\sum_{i+j=k}\frac{k!}{i!j!}P_i(x)P_j(y).
$$
\end{theorem}

We introduce the Cauchy kernel function.

\begin{definition}
The Cauchy kernel $\mathcal G (x)$ is defined by
$$
\mathcal G (x)= \frac{1}{\sigma_n}\frac{\overline{x}}{|x|^{n+1}},\ \ x\in\mathbb{R}^{n+1} \setminus \{0\},\ \
\ \ \sigma_n:=2\frac{2\pi^{(n+1)/2}}{\Gamma((n+1)/2)}.
$$
Moreover, we define for any multi--index $k=(k_1,\dots, k_n)$
$$
\mathcal G_k (x)= \frac {\partial^{|k|}}{\partial x^k}\mathcal G (x).
$$
\end{definition}

Monogenic functions satisfy a generalized integral Cauchy formula (see Theorem $7.12$ in \cite{GHS08}).
\begin{theorem}[The Cauchy formula]\label{t3}
Let $U$ be a bounded domain in $\mathbb{R}^{n+1}$  with smooth  boundary $\partial G$
so that the normal unit vector is orientated outwards.
For the left monogenic functions $f$, defined on an open set that contains $\overline{U}$, we have
$$
f(x)=\int_{\partial G} \mathcal G (y-x)Dy f(y),\ \ \ x\in G.
$$
where
$$
Dy=\sum_{j=0}^n (-1)^j e_j \, dy_0\wedge \dots \wedge dy_{j-1}\wedge dy_{j+1}\wedge\dots\wedge dy_n.
$$
Moreover, for any $x\in G$ and for any multi--index $k=(k_1,\dots, k_n)$, we also have
$$ \frac {\partial^{|k|}}{\partial x^k} f(x)= (-1)^{|k|}\int_{\partial G}\mathcal G_k(y-x)Dy f(y).
 $$
\end{theorem}
If $f$ is left monogenic in a ball centered at the origin and of radius $R$ then for any $|x|<r$ with $0<r<R$ we have
$$
f(x) = \sum_{|k|=0}^{+\infty} V_k(x)a_k,
$$
where the $a_k$'s are Clifford numbers defined by
$$ a_k:=\frac{(-1)^{|k|}}{k!}\int_{|y|=r}\mathcal G_k(y) Dy f(y). $$
By the estimate
\begin{equation}\label{e1}
| \mathcal G_k(x) |\leq \frac{n(n+1)\cdots (n+|k|-1)}{|x|^{n+|k|}}
\end{equation}
the following sharp estimate holds to be true (see \cite{CK02})
\begin{equation}
|a_k|\leq M_g(r) \frac{c(n,k)}{r^{|k|}}.
\end{equation}
where
\begin{equation}\label{cnk}
c(n,k):=\frac{n(n+1)\cdots (n+|k|-1)}{k!}=\frac{(n+|k|-1)!}{(n-1)!\, k!}.
\end{equation}
The first important property that concerns the number $c(n,k)$ is contained in the following lemma (see Lemma $1$ in \cite{CDK07bis}).
\begin{lemma}\label{l4}
For all multi--indexes $k\in(\mathbb N_0)^n\setminus\{ 0\}$ and for all positive integers $n$ we have
$$
 \limsup_{p\to +\infty}\left(\sum_{|k|=p} c(n,k)\right)^{\frac 1p}=n.
$$
\end{lemma}
\begin{definition}
Let $f$ be an entire left monogenic function. Then we say that $f$ is of finite order if there exists $\kappa>0$ such that
$$
M_f(r)< e^{r^k}
$$
for sufficiently large $r$. The greatest lower bound $\rho$ of such numbers $\kappa$ is called order of $f$.
Equivalently, we can define the order as
$$
\rho=\limsup_{r\to\infty}\frac{\ln\ln M_f(r)}{\ln r}.
$$
\end{definition}

\begin{definition}
Let $f$  be an entire left monogenic function of order $\rho$ and let $A>0$ be such that for sufficiently large values of $r$ we have
$$
M_f(r)< e^{Ar^\rho}.
$$
We say that $f$ of order $\rho$ is of type $\sigma$ if $\sigma$ is the greatest lower bound of such numbers and we have
$$
\sigma=\limsup_{r\to\infty}\frac{\ln M_f(r)}{ r^\rho}.
$$
 Moreover, we have
\begin{itemize}
\item
When $\sigma=0$ we say that $f$ is of minimal type.
\item
When $\sigma=\infty$ we say that $f$ is of maximal type.
\item
When $\sigma \in (0,\infty)$ we say that $f$ is of normal type.
\end{itemize}
\end{definition}
The constant functions are said to be of minimal type of order zero. The next two theorems are the generalizations of the Lindel\"of and Pringsheim theorems on the growth of the Taylor coefficients of an entire holomorphic function to the case of an entire monogenic function (see Theorem $1$ in \cite{CDK07bis} and Theorem $1$ in \cite{CDK07}).

\begin{theorem}\label{t1}
For an entire monogenic function $f:\mathbb R^{n+1} \to \mathbb{R}_n$ with a Taylor series representation of the form
$f(x)=\sum_{|k|=0}^{+\infty} V_k(x)a_k$ set
$$
\Pi=\limsup_{|k|\to +\infty}\frac {|k| \log |k|}{-\log|\frac {1}{c(n,k)}a_k|},
$$
then $\rho(f)=\Pi$ where $c(n,k)$ is the number defined in \eqref{cnk}.
\end{theorem}

\begin{theorem}\label{t2}
For an entire monogenic function $f:\mathbb R^{n+1} \to \mathbb{R}_n$ with a Taylor series representation of the form
$f(x)=\sum_{|k|=0}^{+\infty} V_k(x)a_k$ with order $\rho$ ($0<\rho<+\infty$) and set
$$\Pi=\limsup_{|k|\to +\infty} |k|\left( |a_k| \right)^{\frac \rho{|k|}}, $$
then
$$
\displaystyle{\sigma(f)=\frac \Pi{e\rho (f)}}.
$$
\end{theorem}
The next lemma will be useful in the proof of Lemma \ref{l2} and it is a generalization of Lemma \ref{estigam} (III).
\begin{lemma}\label{p1}
Given a multi--index $k\in (\mathbb N_0)^n$, then for any $q\geq 1$ we have
$$
 \Gamma \left( \frac{|k|}{nq} +1\right)^{nq}\leq k!.
$$
\end{lemma}
\begin{proof}
It is sufficient to observe that
\[
\begin{split}
\Gamma \left(\frac{|k|}{q}+1 \right)=\int_0^{+\infty} e^{-t} t^{\frac{|k|}{nq}}\, dt=\int_{0}^{+\infty}\Pi_{i=1}^n e^{\frac tn}t^{\frac{ k_i}{nq}}\,dt & \overset{\textrm{H\"older inequality}}{\leq} \Pi_{i=1}^n \Gamma\left(\frac{k_i}{q}+1\right)^{\frac 1n}\\
& \overset{\textrm{Lemma \ref{estigam}}}{\leq}(k!)^{\frac 1{nq}}.
\end{split}
\]
So we get the statement.
\end{proof}
In the next lemma we introduce the equivalent of the entire Mittag--Leffler functions in one complex variable in the context of the Clifford Analysis. These functions are defined following the way presented at p. 159 in \cite{CDK07} and at p. 772 in \cite{CDK07bis}.
\begin{lemma}\label{l2} Let $x\in \mathbb{R}^{n+1}$ then for any $\alpha,\, \beta\in\mathbb R_{>0}$ the
Mittag-Leffler function
$$
E_{\alpha,\beta}(x)=\sum_{|k|=0}^\infty\frac{c(n, k) V_k(x)}{\Gamma(\alpha |k|+\beta)}
$$
is an entire monogenic function of order $\frac{1}{\alpha}$  and of type $n^{\frac 1\alpha}$ where $c(n,k)$ is the number defined in \eqref{cnk}.
\end{lemma}
\begin{proof}
The ray of convergence of $E_{\alpha, \beta}(x)$ is $+\infty$. For it is sufficient to observe that
$$ \sum_{|k|=0}^\infty\frac{c(n, k) |V_k(x)|}{\Gamma(\alpha |k|+\beta)}\overset{\textrm{Theorem \ref{fpt} IV}}{\leq} \sum_{p=0}^\infty \left(\sum_{|k|=p}^\infty c(n, k)\right) \frac{|x|^{p}}{\Gamma(\alpha p+\beta)}$$
and the conclusion follows by the Cauchy--Hadamard's Theorem for the power series once we note that
$$
\limsup_{p\to +\infty} \left( \left(\sum_{|k|=p}^\infty c(n, k)\right) \frac{1}{\Gamma(\alpha p+\beta)} \right)^{\frac 1p}\overset{\textrm{Lemma \ref{l4}}}{=} 0.
$$
 We prove the remaining part of the lemma for $\beta=1$ since we can deduce the general case by observing that
 $$
 \Gamma(\alpha x+\beta)=\Gamma\left(\alpha\left(x+\frac{\beta-1}{\alpha}\right)+1\right).
   $$
   To prove that the order of $E_{\alpha,1}$ is equal to $\frac 1{\alpha}$ we apply Theorem \ref{t1} with $a_k = \frac{c(n, k)}{\Gamma\left(\alpha |k|+1\right)}$ and the Stirling-De Moivre formula to obtain
\[
 \begin{split}
\rho\left( E_{\alpha,1} \right)&\underset{\textrm{Theorem \ref{t1}}}{=}\limsup_{|k|\to +\infty}\frac {|k| \log |k|}{-\log|\frac {1}{c(n,k)}a_k|}
  \\
 &
 =
\limsup_{|k|\to +\infty}\frac {|k| \log |k|}{\log \Gamma\left(\alpha |k| +1\right)}\underset{\textrm{Stirling-De Moivre}}{=}\frac 1\alpha.
 \end{split}
 \]
To prove that the type of $E_{\alpha, 1}$ is equal to $n^{\frac 1\alpha}$ we apply Theorem \ref{t2} with $a_k = \frac{c(n, k)}{\Gamma\left(\alpha |k|+1\right)}$ to obtain
\[
\begin{split}
\sigma(E_{\alpha,1})&=\frac \alpha e\limsup_{|k|\to +\infty} |k|\left( |a_k| \right)^{\frac {1}{\alpha|k|}}
\\
&
=\frac \alpha e\limsup_{|k|\to +\infty} |k|\left( \frac{c(n,k)}{\Gamma\left( \alpha|k|+1 \right)} \right)^{\frac 1{\alpha |k|}}.
\end{split}
\]
Since by the Lemma \ref{p1}, we have
\[
 \begin{split}
\sup_{|k|=p} \frac{c(n,k)}{\Gamma\left( \alpha|k|+1 \right)}&=\sup_{|k|=p}\frac{(n+|k|-1)!}{(n-1)!\, k!\, \Gamma\left( \alpha|k|+1 \right)}
\\
&
\leq \sup_{|k|=p}\frac{(n+|k|-1)!}{(n-1)!\, \Gamma\left(\frac {|k|}{n}+1\right)^n\, \Gamma\left( \alpha|k|+1 \right)}
\end{split}
\]
with the equality when $|k|$ is a multiple of $n$ and $k=\left(\frac {|k|}{n},\dots, \frac {|k|}{n}\right)$, we can conclude
\[
\begin{split}
& \sigma(E_{\alpha,1})=\frac \alpha e \limsup_{|k|\to +\infty} |k|\left( \frac{(n+|k|-1)!}{(n-1)!\, \Gamma\left(\frac {|k|}{n}+1\right)^n\, \Gamma\left( \alpha|k|+1 \right)}\right)^{\frac 1{\alpha |k|}}\\
& \underset{\textrm{Stirling De-Moivre}}{=} \frac \alpha e \limsup_{|k|\to +\infty} |k|\left( \frac{(n+|k|-1)^{n+|k|-1}}{\left(\frac {|k|}{n}\right)^{|k|}\, \left( \alpha|k|\right)^{\alpha |k|} \exp(-\alpha|k|)}\right)^{\frac 1{\alpha |k|}}= n^{\frac 1 \alpha},
\end{split}
\]
where in the second equality we deleted the terms that do not affect the $\limsup$.
\end{proof}

\begin{definition} Let $p\geq 1$.
We denote by $\mathcal{\!M}^{p}$ the space of entire monogenic
functions with either order lower than $p$ or order equal to $p$ and finite type.
It consists of functions $f$, for which there exist constants $B, C >0$ such that
\begin{equation}\label{ABC}
|f(x)|\leq C e^{B|x|^p}.
\end{equation}
Let $(f_m)_{m\in \mathbb{N} }$, $f_0\in \mathcal{\!M}^p$.
Then $f_m \to f_0$  in $\mathcal{\!M}^p$ if there exists some $B > 0$ such that
\begin{equation}
\lim\limits_{m\rightarrow\infty} \sup_{x\in \mathbb{R}^{n+1}}\Big|(f_m(x)-f_0(x))e^{-B|x|^p}\Big|=0.
\end{equation}
Functions in $\mathcal{\!M}^p$ that are left monogenic will be denoted by $\mathcal{\!M}^p_L$, while
right monogenic will be denoted by $\mathcal{\!M}^p_R$.
\end{definition}
We extend the Lemma $2.2$ in \cite{ACSS18} to the case of the monogenic entire function.
\begin{lemma}\label{l1}
Let $p\geq1$.
A function
$$
f(x)=\sum_{|k|=0}^\infty V_k(x)a_k
$$
 belongs to $\mathcal{\!M}^p$ if and only if there exist constants $C_f,\, b_f,\, \beta>0$ such that
\begin{equation}\label{estfj_1}
|a_k|\leq C_f \frac{b_f^{|k|}c(n,k)}{\Gamma\left (\frac{|k|}{p}+\beta\right)}.
\end{equation}
Furthermore, a sequence $f_m$ in $\mathcal{\!M}^p$ tends to zero if and only if $C_{f_m}\to0$ and $b_{f_m} < b$ for some $b>0$ where $c(n,k)$ is the number defined in \eqref{cnk}.
\end{lemma}
\begin{proof}
$(\Rightarrow)$ By the Theorem \ref{t3} we have
$$ \partial_x^k f(x)=(-1)^{|k|}\int_{\partial B(x, s|x|)}\mathcal G_k(y-x)Dy f(y) $$
where $s>0$ is a constant to be determined later. In view of the estimate \eqref{e1} and since $f\in\mathcal M^p$, we have
\[
\begin{split}
|\partial_x^k f(x)| & \leq k! \, \frac{c(n,k)}{(s|x|)^{|k|}}\sup_{|\zeta-x|=s|x|}|f(\zeta)|\\
&\leq k! \, \frac{c(n,k)}{(s|x|)^{|k|}} M_f((1+s)|x|)\\
&\leq C_f k!\, \frac{c(n,k)}{(s|x|)^{|k|}} \exp(B(1+s)^p|x|^p)\\
&\leq C_f k!\, \frac{c(n,k)}{(s|x|)^{|k|}} \exp(B2^ps^p|x|^p)\exp(B2^p|x|^p)
\end{split}
\]
where the last inequality is due to the estimate: $(1+s)^p\leq 2^p(1+s^p)$. We define
$$ g(s):= \frac{\exp(B2^p s^p |x|^p)}{(s|x|^{|k|})} $$
and we note that this function gets its minimum at
$$
s_0:= \frac 12\left(\frac{|k|}{pB}\right)^{\frac 1p}\frac{1}{|x|}.
$$
Thus we have
$$
g(s_0)=\exp\left(\frac{|k|}{p}\right)\left(\frac{2pB}{|m|}\right)^{\frac{|m|}{p}},
$$
and
$$
 |\partial_x^k f(x)|\leq C_f\, k!\, c(n,k) \left[\left( 2epB \right)^{\frac 1p}\right]^{|k|}|k|^{-\frac{|k|}{p}}\exp(B2^{p}|x|^p).
$$
We set $b=\left(2epB\right)^{\frac 1p}$ and by the maximum modulus principle we have
\[
\begin{split}
|a_k|=\frac{|\partial_x^kf(0)|}{k!} & \leq \frac{\sup_{|x|=r}|\partial_x^kf(x)|}{k!}
\\
&
\leq C_f\, c(n,k)\, b^{|k|}\, |k|^{-\frac{|k|}{p}}\exp(B2^p r^p)
\\
&
\leq 2 C_f\, c(n,k)\, b^{|k|}\, |k|^{-\frac{|k|}{p}}
\\
&
\leq 2C_f'\, c(n,k)\, b^{|k|}\, (|k|!)^{-\frac{1}{p}}
\\
&\overset{\textrm{Lemma \ref{estigam} (III)}}{\leq} 2C_f' c(n,k) \frac{b^{|k|}}{\Gamma\left(\frac{|k|}{p}+1\right)}.
\end{split}
\]
$(\Leftarrow)$ The other direction is a consequence of the properties of the Mittag-Leffler function described in Lemma \ref{l2}
\end{proof}

To define a class of operators that act over $\mathcal M_L^p$ with image in the same space, it is useful to introduce the left (resp. right) C-K product between left (resp. right) monogenic entire functions (see p. 114 in \cite{BDS82}).
\begin{definition}
Let $f,g\in \mathcal M_{L}$ be entire functions (resp. $f,g\in \mathcal M_{R}$).
Using their Taylor series representation
$$ f(x)=\sum_{|k|=0}^{+\infty} V_k(x) f_k\quad \textrm{(resp. $f(x)=\sum_{|k|=0}^{+\infty} f_k V_k(x)$ )} $$
and
$$ g(x)=\sum_{|k|=0}^{+\infty} V_k(x) g_k\quad \textrm{(resp. $g(x)=\sum_{|k|=0}^{+\infty} g_k V_k(x)$ )} $$
 we define
$$f \odot_L g:= \sum_{|k|=0}^{+\infty} \sum_{|j|=0}^{+\infty} V_{k+j}(x) f_k g_j\quad\textrm{(resp. $f \odot_R g:= \sum_{|k|=0}^{+\infty} \sum_{|j|=0}^{+\infty} f_k g_j V_{k+j}(x)$)}.$$
\end{definition}

\begin{remark}
At p. 114 in \cite{BDS82}, the definition of C-K product is given between Fueter polynomial. Here we adapt that definition for the polynomial $V_k(x)$ introduced in the Definition \ref{fp}.
\end{remark}

\section{Infinite order differential operators  on monogenic functions}\label{INMON}

In this section, using the definition of the C-K product of two monogenic entire functions
we define suitable classes of infinite order differential operators in the monogenic setting.

\begin{definition}\label{o1}
Let $p\geq 1$ and set $\mathbb{N}_0=\mathbb{N} \cup\{0\}$.
\begin{itemize}
\item
Let $(u_m)_{m\in (\mathbb{N}_0)^n }:\mathbb{R}^{n+1}\to \mathbb{R}_n$ be  entire functions that belong to $\mathcal{M}_L$.
We define the set $\mathbf{DM}^L_{p,0}$ of formal
operators defined by
$$
U_L(x,\pp_{x})f(x):=\sum_{|m|=0}^\infty u_m(x)\odot_L \pp_{x}^m f(x),
$$
for entire functions $f$ in $\mathcal{M}_L$ where $\pp_{x}^m:= \pp_{x_1}^{m_1}\dots\pp_{x_n}^{m_n}$ (in particular no derivatives along the $x_0$-direction appear).
\item
Let $(u_m)_{m\in (\mathbb{N}_0)^n}:\mathbb{R}^{n+1}\to \mathbb{R}_n$ be  entire functions that belong to $\mathcal{M}_R$.
We define the set $\mathbf{DM}^R_{p,0}$ of formal
operators defined by
$$
U_R(x,\pp_{x})f(x):=\sum_{|m|=0}^\infty u_m(x)\odot_R \pp_{x}^m f(x)
$$
for entire functions $f$ in $\mathcal{M}_R$ where $\pp_{x}^m:= \pp_{x_1}^{m_1}\dots\pp_{x_n}^{m_n}$ (in particular no derivatives along the $x_0$-direction appear).
\end{itemize}
The entire functions $(u_m)_{m\in (\mathbb{N}_0)^n}$ in $\mathcal{M}_L$ (resp. in $\mathcal{M}_R$) satisfy the additional condition: There exists a constant $B>0$ such
 that for every $\varepsilon>0$ there exists a constant $C_\varepsilon>0$ for which
 \begin{equation}\label{e1bis}
 |u_m(x)|\leq C_\varepsilon \frac{\varepsilon^{|m|}}{(|m|!)^{1/q}}\exp(B|x|^p), \ \ \ {\rm for \ all}  \ \ \ m\in (\mathbb{N}_0)^n,
 \end{equation}
where $1/p+1/q=1$ and $1/q=0$ when $p=1$.
\end{definition}

\begin{remark}\label{r1}
If we consider the Taylor expansion of the $u_m$'s then we can write them as
$$u_m(x)=\sum_{|j|=0}^{+\infty}V_j(x)a_j^m$$
and, thanks to the Lemma \ref{l1}, the coefficients $a_j^m$'s satisfy the estimate
$$
 |a_j^m|\leq C_\epsilon \frac{\epsilon^{|m|} (b_{u_m})^{|j|} c(n,j)}{(m!)^{\frac{1}{q}}\Gamma\left(\frac {j}{p}+1\right)},
$$
where $c(n,j)$ is the number defined in \eqref{cnk}.
\end{remark}

We are now in the position to state and proof the main result of this section (the analogue of the Theorem $2.4$ in \cite{ACSS18}).
\begin{theorem}\label{t4}
Let $p\geq 1$ and let $\mathbf{DM}^L_{p,0}$ and $\mathbf{DM}^R_{p,0}$ be the sets of formal operators in Definition \ref{o1}.
\begin{itemize}
\item[(I)]
Let  $U_L(x,\pp_{x})\in \mathbf{DM}^L_{p,0}$ and let $f\in \mathcal{M}^p_L$, then $U_L(x,\pp_{x})f\in \mathcal{M}^p_L$ and
the operator $U_L(x,\pp_{x})$ acts continuously on $\mathcal{M}^p_L$, i.e., if $f_m\in \mathcal{M}^p_L$
and  $f_m\to 0$ in $\mathcal{M}^p_L$ then we have $U_L(x,\pp_{x})f_m\to 0$ in $\mathcal{M}^p_L$.
\item[(II)]
Let  $U_R(x,\pp_{x})\in \mathbf{DM}^R_{p,0}$ and let $f\in \mathcal{M}^p_R$, then $U_R(x,\pp_{x})f\in \mathcal{M}^p_R$ and
the operator $U_R(x,\pp_{x})$ acts continuously on $\mathcal{M}^p_R$, i.e., if $f_m\in \mathcal{M}^p_R$
and  $f_m\to 0$ in $\mathcal{M}^p_R$ then we have $U_R(x,\pp_{x})f_m\to 0$ in $\mathcal{M}^p_R$.
\end{itemize}
\end{theorem}
\begin{proof}
We write in details only the proof of the statement (I) since the proof of the statement (II) follows from minor changes. Since $u_m,\, f\in\mathcal M^p_L$ are entire functions we can rewrite them using their Taylor expansion series
$$ f(x)=\sum_{|k|=0}^\infty V_k(x)f_k\quad \textrm{and} \quad u_m(x)=\sum_{|k|=0}^\infty V_k(x)a^m_k $$
where $a_k^m,\, f_k\in \mathbb{R}_n$ for any multi-indexes $m,\, k\in(\mathbb N_0)^n$. According to the Definition \ref{o1}, we have
\[
\begin{split}
 U_L(x,\pp_{x}) f(x)  &= \sum_{|m|=0}^\infty  u_m(x)\odot_L \pp_{x}^m f(x)
 \\
 &
 =  \sum_{|m|=0}^\infty u_m(x)\odot_L  \sum_{|k|=0}^\infty \pp_x^m (V_k(x)) f_k
 \\
&
=  \sum_{|m|=0}^\infty u_m(x)\odot_L  \sum_{|k|=0}^\infty \frac{(m+k)!}{k!} V_k(x) f_{m+k}
 \\
 &
 = \sum_{|m|=0}^\infty  \sum_{|k|=0}^\infty\sum_{|j|=0}^{\infty} \frac{(m+k)!}{k!} V_{k+j}(x)  a_j^m f_{m+k}.
\end{split}
\]
Since the $V_k(x)$'s are paravectors, using the classical inequality: $|xy|\leq 2^{\frac n2}|x| |y|$ for any $x,\, y\in \mathbb{R}_n$, we have that
 $$
  \left|U_L(x,\pp_{x}) f(x) \right|\leq 2^{\frac n2} \sum_{|m|=0}^\infty \sum_{|k|=0}^\infty \sum_{|j|=0}^\infty \frac{(m+k)!}{k!} |a_j^m| |f_{m+k}| |V_{k+j}(x)|.
   $$
We observe that if we define
$$ \underline y= y_1 e_1+\dots+ y_ne_n := \left(\sqrt{x_1^2+x_0^2}\right)e_1+\dots+\left(\sqrt{x_n^2+x_0^2}\right)e_n $$
then we have:
$$
 |V_k(x)|\leq V_k\left(\underline y\right).
$$
Since
$$V_k\left(\underline y\right)=k!\, \Pi_{i=1}^n \left(y_i\right)^{k_i}$$
we get
\begin{equation}\label{i1}
 |V_{k+j}(x)|\leq V_{k+j}(\underline y)= \frac{(k+j)!}{k!\, j!} V_k\left(\underline y\right)V_j\left(\underline y\right) \overset{\textrm{Lemma \ref{estigam} (I)}}{\leq} 2^{|k|+|j|}V_k\left(\underline y\right) V_j\left(\underline y\right).
 \end{equation}
Moreover, since $|\underline y|\leq \sqrt n |x|$ and using the Remark \ref{r1} in combination with the Lemma \ref{l1}, we have
\[
\begin{split}
\left|U_L(x,\pp_{x}) f(x) \right| & \leq 2^{\frac n2} \sum_{|m|=0}^\infty \sum_{|k|=0}^\infty \left (\sum_{|j|=0}^\infty 2^{|j|} |a_j^m|V_j(\underline y)\right) \frac{2^{|k|} (m+k)!}{k!} |f_{m+k}| V_{k}(\underline y)\\
& \leq 2^{\frac n2} C_\varepsilon \sum_{|m|=0}^\infty \sum_{|k|=0}^\infty \frac{\varepsilon^{|m|}}{(|m|!)^{1/q}}\exp(B|x|^p) \frac{2^{|k|} (m+k)!}{k!} |f_{m+k}| V_{k}(\underline y).
\end{split}
\]
Using Lemma \ref{l1} and the estimates \eqref{e1bis}, we obtain
\begin{equation}\label{fe4}
\begin{split}
& \left|U_L(x,\pp_{x}) f(x) \right| \leq 2^{\frac n2} C_f C_\varepsilon \sum_{|m|=0}^\infty \sum_{|k|=0}^\infty \frac{(m+k)!\, \varepsilon^{|m|} b_f^{|m|}}{k!\,(|m|!)^{1/q}}\frac{\,(2b_f)^{|k|}c(n,m+k)}{\Gamma\left (\frac{|m+k|}{p}+1\right)} \exp(B|x|^p) V_k(\underline y)  \\
& \overset{\textrm{Lemma \ref{estigam} (I)+(III)}}{\leq} 2^{\frac n2} C_f C_\varepsilon \sum_{|m|=0}^\infty \sum_{|k|=0}^\infty \frac{ m! \, (2\varepsilon b_f)^{|m|}}{\Gamma\left (\frac{|m|}{q}+1\right)}\frac{\, (4b_f)^{|k|}c(n,m+k)}{\Gamma\left (\frac{|m|}{p}+\frac 12\right)\Gamma\left (\frac{|k|}{p}+\frac 12\right)} \exp(B|x|^p)V_k(\underline y)\\
& \leq 2^{\frac n2} C_f C_\varepsilon \sum_{|m|=0}^\infty  \frac{(\varepsilon 2b_f)^{|m|}\,c(n,m)\,  m!\, \exp(B|x|^p)}{\Gamma\left (\frac{|m|}{q}+1\right) \Gamma\left (\frac{|m|}{p}+\frac 12\right)} \sum_{|k|=0}^\infty \frac{c(n, m+k)}{c(n,m)c(n,k)}\frac{(4b_f)^{|k|}c(n,k)\, V_k(\underline y)}{\Gamma\left (\frac{|k|}{p}+\frac 12\right)}\\
& \leq 2^{\frac{3n}{2}-1} (n-1)!\, C_f C_\varepsilon \sum_{|m|=0}^\infty  \frac{(\varepsilon 4b_f)^{|m|}\, c(n,m) \, m!\, \exp(B|x|^p)}{\Gamma\left (\frac{|m|}{q}+1\right) \Gamma\left (\frac{|m|}{p}+\frac 12\right)} \sum_{|k|=0}^\infty \frac{(8b_f)^{|k|}c(n,k) \, V_k(\underline y)}{\Gamma\left (\frac{|k|}{p}+\frac 12\right)},
\end{split}
\end{equation}
where the last inequality is due to the following estimate:
$$
\frac{c(n, m+k)}{c(n,m)c(n,k)}=\frac{(n+|m|+|k|-1)!\, ((n-1)!)^2\, k!\, m!}{(n-1)! (m+k)! (n+|m|-1)!\, (n+|k|-1)!}\leq (n-1)!\, 2^{n+|m|+|k|-1}.
$$
We observe that:
\[
\begin{split}
&\frac{m!}{\Gamma\left (\frac{|m|}{q}+1\right) \Gamma\left (\frac{|m|}{p}+\frac 12\right)}\leq \frac{|m|!}{\Gamma\left (\frac{|m|}{q}+1\right) \Gamma\left (\frac{|m|}{p}+\frac 12\right)}\\
&\overset{\textrm{Stirling-DeMoivre}}{\sim} \frac{|m|^{|m|} \sqrt{|m|} \exp(-|m|)}{\left (\frac{|m|}{q}\right)^{\frac{|m|}{q}} \left (\frac{|m|}{p}-\frac 12\right)^{\frac{|m|}{p}-\frac 12} |m|\exp(-|m|)}\lesssim \left (p^{\frac 1p}q^{\frac 1q}\right)^{|m|}.
\end{split}
\]
In particular the previous inequality implies that:
\begin{equation}\label{fe3}
\sum_{|m|=0}^\infty\frac{(\varepsilon 4b_f)^{|m|}\, c(n,m) \, m!}{\Gamma\left (\frac{|m|}{q}+1\right) \Gamma\left (\frac{|m|}{p}+\frac 12\right)}\lesssim \sum_{|m|=0}^\infty \left (p^{\frac 1p}q^{\frac 1q}\varepsilon 4b_f \right)^{|m|}\, c(n,m).
\end{equation}
By the Lemma \ref{l4} and since $\epsilon>0$ can be chosen small enough, using the Cauchy-Hadamard's Theorem for the power series, we have that the previous series converges. Thus there exists a constant $C'>0$ such that
\begin{equation}\label{fe5}
\sum_{|m|=0}^\infty\frac{(\varepsilon 4b_f)^{|m|}\, c(n,m) \, m!}{\Gamma\left (\frac{|m|}{q}+1\right) \Gamma\left (\frac{|m|}{p}+\frac 12\right)}\leq C'.
\end{equation}
By the Lemma \ref{l1} there exist two constants: $B'>0$ and $C''>0$ such that:
\begin{equation}\label{fe2}
\sum_{|k|=0}^\infty \frac{(8b_f)^{|k|}c(n,k)}{\Gamma\left (\frac{|k|}{p}+\frac 12\right)}V_k(\underline y) \leq C''\exp\left(B'|x|^p\right).
\end{equation}
In conclusion by the estimates \eqref{fe5} and \eqref{fe2} we have proved that
$$
\left|U_L(x,\pp_{x}) f(x) \right| \leq 2^{2n-1} (n-1)!\, C'' C_f C_\varepsilon C' \exp\left((B+B')|x|^p\right)
$$
which means that $U_L(x,\pp_{x}) f(x)\in \mathcal M^p_L$ and also that $U_L(x,\pp_{x})$ is continuous over $\mathcal M^p_L$ i.e. $U_L(x,\pp_{x}) f(x)\to 0$ as $f\to 0$ or, equivalently, $C_f\to 0$.
\end{proof}
\begin{remark}
If in the Definition \ref{o1} we use the standard product of the Clifford Algebra instead of the C-K product to define  $U_L(x,\pp_{x})$ (or $U_R(x,\pp_{x})$), the resulting operator does not preserve the monogenicity although it remains a continuous operator from $\mathcal M_L^p$ (or $\mathcal M_R^p$) to $C^0(\mathbb R^{n+1}, \mathbb{R}_n)$. This can be seen starting from the inequality
\[
\begin{split}
\left| \sum_{|m|=0}^\infty u_m(x)\partial_{x}^m f(x) \right| & \leq 2^{\frac n2} \sum_{|m|=0}^\infty \sum_{|k|=0}^\infty \frac{(m+k)!}{k!} |u_m(x)| |f_{m+k}| V_k(\underline y)\\
& \leq 2^{\frac n2} \sum_{|m|=0}^\infty \sum_{|k|=0}^\infty \frac{(m+k)!}{k!} \frac{\epsilon ^q}{(|m|!)^{\frac 1q}}\exp(B|x|^p) |f_{m+k}| V_k(\underline y)
\end{split}
\]
and applying to the last term in the right the same estimates used in \eqref{fe4}, \eqref{fe5} and \eqref{fe2}.
  \end{remark}

\section{Concluding remarks}

\medskip
The two hyperholomorphic function theories have several applications
 both in Mathematics and in Physics. Precisely,
 associated with slice hyperholomorphic functions, \cite{ACSBOOK2, MR2752913,BOOKGS,GSSb}
 it is possible to define the spectral theory on the $S$-spectrum \cite{CGKBOOK,MR2752913} that has applications in
 quaternionic quantum mechanics \cite{adler,JONAQS}, in fractional diffusion processes \cite{Deniz,64FRAC,FJBOOK,CPP},
 characteristic operator functions \cite{COF}, the spectral theorem for quaternionic normal operators \cite{ack}, the perturbations theory
  \cite{CCKSpert}, Schur analysis \cite{ACSBOOK} and
 other fields are under investigation.
 The monogenic function theory is associated with harmonic analysis in higher dimension.
 Moreover,  there exists a functional calculus based on the Cauchy formula from which
  the monogenic spectrum was defined, see \cite{jefferies}, the function theory has applications to boundary value problems, see \cite{GSBOOK}.

\medskip
 The notion of superoscillatory functions first appears in a series of works
of Y. Aharonov, M. V. Berry and co-authors, see \cite{aav, abook, av, berry2,berry,b1,b4}. In this context, there are good physical reasons for such a behavior,
but the discoverers pointed out the apparently paradoxical nature of such functions,
thus opening the way for a more thorough mathematical analysis of the phenomenon.
In a series of recent papers there are some systematic
study of superoscillations from the mathematical point of view, see
\cite{ABCS19,acsst1,acsst3, acsst6, JFAA, QS1, KGField, QS2, Jussi, QS3}
and see also \cite{BerryMILAN,genHYP,kempfQS}.

\medskip
The theory of superoscillations appears in various questions, namely extension of positive definite functions,
interpolation of polynomials and also of $R$-functions and have applications to signal theory and prediction
theory of stationary stochastic processes.
Thing in some of this area are still under investigation as one can see in the paper \cite{ACSFOUR}.

\medskip
 The relation between the two classes of hyperholomorphic function can be seen
by rephrasing in modern language  the Sce's theorem who generalize a theorem of
Fueter for the quaternionic setting in a non trivial and original way to Clifford valued functions,
see the book \cite{bookSce} for an overview of this profound theorem of complex and hypercomplex analysis. Precisely,
 let $\tilde{f}(z) = f_0(u,v)+if_1(u,v)$ be a  holomorphic function
defined in a domain (open and connected) $D$ such that $f_0(u,v)+if_1(u,v)$ satisfy the conditions of the type (\ref{CCondmon}) and let
$$
\Omega _D= \{x =x_0+\underline{x}\ \ |\ \  (x_0, |\underline{x}|) \in D\}
$$
be the open set induced by $D$ in $\mathbb{R}^{n+1}$.

(Step I)
The  map $T_{1}$, defined by:
 $
f(x)=T_{1}(\tilde{f}):=\textcolor{black}{f_1(x_0,|\underline{x}|)+\frac{\underline{x}}{|\underline{x}|}f_1(x_0,|\underline{x}|)}
$
takes the holomorphic functions $\tilde{f}(z)$ and induces the Clifford-valued function $f(x)$ that is slice hyperholomorphic.

(Step II)
The map $T_{2}:=\Delta_{n+1}^{\frac{n-1}{2}}$ where $\Delta_{n+1}$ is the laplacian in $n+1$ dimensions
applied to a slice hyperholomorphic function, i.e.,
$
\breve{f}(x):=\textcolor{black}{T_{2}} \Big(f(x)\Big),
$
defines a function that
is in the kernel of the Dirac operator, i.e.,
$
D\breve{f}(x)=0 \ \ \ {\rm on} \ \ \Omega_D,
$
where $D$ is the Dirac operator.

\medskip
In the Sce's theorem
$\tilde{f}(z) = f_0(u,v)+if_1(u,v)$ is a holomorphic function where $f_0(u,v)$ and $f_1(u,v)$ are real-valued functions, but considering
  slice monogenic functions, where $f_0(u,v)$ and $f_1(u,v)$ are Clifford-valued functions, and applying
the map $T_{2}:=\Delta_{n+1}^{\frac{n-1}{2}}$,
in step (II) of the Sce's construction, we get a function that is in the kernel of the Dirac operator.
The case when $\frac{n-1}{2}$ is fractional has been studied by T. Qian; for the precise history of this theorem in the notes of the book \cite{bookSce}.

\end{document}